\documentclass[11pt]{amsart}
\usepackage[dvipsnames,usenames]{color}
\usepackage{hyperref}
\usepackage{graphicx}
\usepackage{epsfig}
\usepackage[latin1]{inputenc}
\usepackage{amsmath}
\usepackage{amsfonts}
\usepackage{amssymb}
\usepackage{amsthm}
\usepackage{amscd}
\usepackage{verbatim}
\usepackage{subfigure}
\usepackage{pinlabel}
\usepackage{stmaryrd}
\usepackage{enumerate, enumitem}

\usepackage[textsize=scriptsize]{todonotes}

\usepackage{tikz}
\usetikzlibrary{cd}
\usetikzlibrary{arrows}
\usetikzlibrary{decorations.pathreplacing}
\usepackage{verbatim}

\newtheorem{theorem}{Theorem}[section]

\newtheorem{lemma}[theorem]{Lemma}
\newtheorem{proposition}[theorem]{Proposition}

\theoremstyle{definition}

\newtheorem{convention}[theorem]{Convention}
\theoremstyle{remark}
\newtheorem{remark}[theorem]{Remark}

\def\Q{\mathbb{Q}}

\def\Z{\mathbb{Z}}

\def\HFred{\operatorname{HF_{red}}}

\def\conn {\mathbin{\#}}

\def\spincs {\mathfrak{s}}
\def\spinct {\mathfrak{t}}

\def\RP{\mathbb{R} \mathrm{P}}

\definecolor{darkgreen}{rgb}{0,0.5,0}
\definecolor{purple}{rgb}{0.5,0,0.5}

\newcommand{\abs}[1] {\left\lvert #1 \right\rvert}

\newcommand{\gen}[1] {\langle #1 \rangle}
\def\ul {\underline}

\title{Simply-connected, spineless 4-manifolds}

\subjclass[2013]{}

\author[Adam Simon Levine]{Adam Simon Levine}
\thanks{The first author was partially supported by NSF grant DMS-1707795.}
\address{Department of Mathematics, Duke University, Durham, NC 27708}
\email{alevine@math.duke.edu}

\author[Tye Lidman]{Tye Lidman}
\thanks{The second author was partially supported by NSF grant DMS-1709702.}
\address{Department of Mathematics, North Carolina State University, Raleigh, NC 27607}
\email{tlid@math.ncsu.edu}

\numberwithin{equation}{section}

\begin{document}

\maketitle

\begin{abstract}
We construct infinitely many smooth $4$-manifolds which are homotopy equivalent to $S^2$ but do not admit a spine, i.e., a piecewise-linear embedding of $S^2$ which realizes the homotopy equivalence.  This is the remaining case in the existence problem for codimension-$2$ spines in simply-connected manifolds.  The obstruction comes from the Heegaard Floer $d$ invariants.
\end{abstract}

\section{Introduction}

Given an $m$-dimensional, piecewise-linear, compact manifold $M$ which is homotopy equivalent to some closed manifold $N$ of dimension $n<m$, a \emph{spine} of $M$ is a piecewise-linear embedding $N \to M$ which is a homotopy equivalence. Such an embedding is not required to be locally flat. We call $M$ \emph{spineless} if it does not admit a spine.

In this paper, we prove:

\begin{theorem}\label{thm:spineless}
There exist infinitely many smooth, compact, spineless $4$-manifolds which are homotopy equivalent to $S^2$.
\end{theorem}

By way of background, Browder \cite{BrowderEmbedding}, Casson, Haefliger \cite{HaefligerKnotted}, Sullivan, and Wall \cite{WallSurgery} showed that when $m-n > 2$, any homotopy equivalence from $N$ to $M$ can be perturbed into a spine. When $m-n=2$, Cappell and Shaneson \cite{CappellShanesonPL} showed that the same is true for any odd $m\ge 5$, and for any even $m \ge 6$ provided that $M$ and $N$ are simply-connected; they also produced examples of non-simply-connected, spineless manifolds for any even $m \ge 6$ \cite{CappellShanesonSpineless}. (See \cite{ShanesonSpines} for a summary of their results.) In dimension $4$, Matsumoto \cite{MatsumotoSpine} produced an example of a spineless $4$-manifold homotopy equivalent to the torus; the proof relies on higher-dimensional surgery theory. However, the question of finding spineless, simply-connected $4$-manifolds has remained open until now. (It appears in Kirby's problem list \cite[Problem 4.25]{Kirby}.)

The proof of the theorem proceeds in two parts.  The first is to give an obstruction to a spine in a
PL $4$-manifold homotopy equivalent to $S^2$ coming from Heegaard Floer homology.  This obstruction
only depends on the boundary of the 4-manifold and the sign of the intersection form. The second
step is to construct the manifolds homotopy equivalent to $S^2$ that fail the obstruction.

\subsection*{Acknowledgments}

We are grateful to Weimin Chen, \c{C}a\u{g}r{\i}  Karakurt, and Yukio Matsumoto for bringing this problem to our attention, and to Marco Golla, Josh Greene, Jen Hom, and Danny Ruberman for helpful conversations.

\section{Obstruction} \label{sec:d-invt}

In order to prove Theorem~\ref{thm:spineless}, we use an obstruction coming from Heegaard Floer homology. Recall that for any rational homology sphere $Y$ and any spin$^c$ structure $\spincs$ on $Y$, Ozsv\'ath and Szab\'o \cite{OSabsgr} define the \emph{correction term} $d(Y,\spincs) \in \Q$, which is invariant under spin$^c$ rational homology cobordism. To state our obstruction, we first establish the following notational convention.

\begin{convention} \label{conv:spinc}
Suppose $X$ is a smooth, compact, oriented $4$-manifold with $H_*(X) \cong H_*(S^2)$, and let $n$ denote the self-intersection number of a generator of $H_2(X)$. Let $Y = \partial X$, which has $H_1(Y) \cong H^2(Y) \cong \Z/n$. Fix a generator $\alpha \in H_2(X)$. For $i \in \Z$, let $\spinct_i$ denote the unique spin$^c$ structure on $X$ with
\[
\gen{c_1(\spinct_i), \alpha} + n = 2i.
\]
Let $\spincs_i = \spinct_i |_Y$; this depends only on the class of $i$ mod $n$. We will often treat the subscript of $\spincs_i$ as an element of $\Z/n$.

Conjugation of spin$^c$ structures swaps $\spinct_i$ with $\spinct_{n-i}$ and $\spincs_i$ with
$\spincs_{n-i} = \spincs_{-i}$. In particular, $\spincs_0$ is self-conjugate, as is $\spincs_{n/2}$
if $n$ is even. Choosing the opposite generator for $H_2(X)$ likewise replaces each $\spinct_i$ or
$\spincs_i$ with its conjugate. Because of the conjugation symmetry of Heegaard Floer homology, all
statements below are insensitive to this choice.

Finally, when $n \ne 0$, we have
\begin{equation} \label{eq:d-mod2}
d(Y,\spincs_i) \equiv \frac{(2i-n)^2  - \abs{n} }{4n} \pmod {2\Z}
\end{equation}
by \cite[Theorem 1.2]{OSabsgr}.
\end{convention}

Our obstruction to the existence of a spine comes from the following theorem:

\begin{theorem}\label{thm:obstruction}
Let $X$ be any smooth, compact, oriented $4$-manifold with $H_*(X) \cong H_*(S^2)$, with a generator
of $H_2(X)$ having self-intersection $n > 1$, and let $Y = \partial X$. If a generator of $H_2(X)$
can be represented by a piecewise-linear embedded $2$-sphere (e.g., if $X$ admits an $S^2$ spine),
then for each $i \in \{0, \dots, n-1\}$, 
\begin{equation} \label{eq:obstruction}
d(Y, \spincs_i) - d(Y, \spincs_{i+1}) = 
\begin{cases}
\dfrac{n-2i-1}{n} \text{ or } \dfrac{-n-2i-1}{n} & \text{if } 0 \le i \le \dfrac{n-2}{2} \\
0 & \text{if $n$ is odd and } i = \dfrac{n-1}{2} \\
\dfrac{n-2i-1}{n} \text{ or } \dfrac{3n-2i-1}{n} & \text{if } \dfrac{n}{2} \le i \le n-1.
\end{cases}
\end{equation}
In particular, for any $i$, we have
\begin{equation} \label{eq:Y-abs}
\abs{d(Y, \spincs_i) - d(Y, \spincs_{i+1})} \le \frac{2n-1}{n}.
\end{equation}
\end{theorem}

It is easy to verify that \eqref{eq:Y-abs} follows as an easy consequence of \eqref{eq:obstruction}.

For any knot $K \subset S^3$, let $X_n(K)$ denote the trace of $n$-surgery on $S^3$, i.e., the manifold obtained by attaching an $n$-framed $2$-handle to the $4$-ball along a knot $K \subset S^3$. Note that $X_n(K)$ is homotopy equivalent to $S^2$ and has a spine obtained as the union of the cone over $K$ in $B^4$ with the core of the $2$-handle.

\begin{lemma} \label{lemma:S3-surgery}
For any knot $K \subset S^3$ and any $n>0$, the manifold $Y = S^3_n(K)$ satisfies the conclusions of Theorem \ref{thm:obstruction}.
\end{lemma}

\begin{proof}
Associated to any knot $K \subset S^3$, Ni and Wu \cite[Section 2.2]{NiWu} defined a sequence of nonnegative integers $V_i(K)$, which are derived from the knot Floer complex of $K$. (See also \cite{RasmussenThesis}.) By \cite[Equation 2.3]{HomWu4Genus}, these numbers have the property that
\begin{equation} \label{eq:Vi-noninc}
V_i(K)-1 \le V_{i+1}(K) \le V_i(K);
\end{equation}
that is, the sequence $(V_i(K))$ is non-increasing and only decreases in increments of
$1$. Ni and Wu proved that for each $i=0, \dots, n-1$, we have
\begin{equation} \label{eq:NiWu}
d(S^3_n(K), \spincs_i) = \frac{(2i-n)^2-n}{4n} - 2 \max\{ V_i(K), V_{n-i}(K)\}.
\end{equation}
(The first term in \eqref{eq:NiWu} is the $d$ invariant of the lens space $L(n,1)$ in a particular spin$^c$ structure; see \cite[Proposition 4.8]{OSabsgr}.)

For $0 \le i \le \frac{n-2}{2}$, we then compute:
\begin{align*}
d(S^3_n(K), \spincs_i) - d(S^3_n(K), \spincs_{i+1})
&= \frac{(2i-n)^2 - (2i+2-n)^2}{4n} - 2(V_i(K)-V_{i+1}(K)) \\
&= \frac{n-2i-1}{n} - 2(V_i(K)-V_{i+1}(K)) \\
&= \frac{n-2i-1}{n} \text{ or } \frac{-n-2i-1}{n}.
\end{align*}
(The last line follows from the fact that $V_i(K)-V_{i+1}(K)$ equals either $0$ or $1$.)

If $\frac{n}{2} \le i \le n-1$, then
\begin{align*}
d(S^3_n(K), \spincs_i) - d(S^3_n(K), \spincs_{i+1}) &= d(S^3_n(K), \spincs_{n-i}) - d(S^3_n(K), \spincs_{n-i-1}),
\end{align*}
and we may apply the previous case using $n-i-1$ in place of $i$.

Finally, in the special case where $n$ is odd and $i=\frac{n-1}{2}$, it is easy to compute that
\[
d(S^3_{n}(K), \spincs_i) = d(S^3_{n}(K), \spincs_{i+1}) = \frac{1-n}{4n} - 2V_{(n-1)/2}(K),
\]
so the difference is $0$, as required.
\end{proof}

\begin{proof}[Proof of Theorem \ref{thm:obstruction}]
Suppose $S$ is a PL embedded sphere representing a generator of $H_2(X)$. We may assume that $S$ has a single singularity modeled on the cone of a knot $K \subset S^3$ and is otherwise smooth. Therefore, $S$ has a tubular neighborhood diffeomorphic to $X_n(K)$.  To see this, observe that a neighborhood of the cone point is a copy of $B^4$ and the rest of the neighborhood then makes up a 2-handle attached along $K$.  That the framing is $n$ follows from the fact that the intersection form of $X$ is $(n)$.  The complement of the interior of this neighborhood is a homology cobordism between $S^3_n(K)$ and $Y$; moreover, for each $i \in \Z/n$, the spin$^c$ structures labeled $\spincs_i$ on $S^3_n(K)$ and $Y$ as in Convention~\ref{conv:spinc} are identified through this cobordism. In particular, $d(Y,\spincs_i) = d(S^3_n(K), \spincs_i)$. By Lemma \ref{lemma:S3-surgery}, we deduce that the conclusions of the theorem hold for $Y$.
\end{proof}

\begin{remark} \label{rmk:NiWu-fail}
For surgery on a knot $K$ in an arbitrary homology sphere $Y$, the analogue of the Ni--Wu formula \eqref{eq:NiWu} need not hold. Instead, just as in our paper with Hom \cite[Lemma 2.2]{HomLevineLidmanConcordance}, one can prove an inequality
\begin{equation} \label{eq:NiWu-ineq}
-2N_Y \leq d(Y_n(K), \spincs_i) - d(Y) - \frac{(2i-n)^2-n}{4n} + 2 \max\{ V_i(K), V_{n-i}(K)\} \leq
0
\end{equation}
where
\[
N_Y = \min \{k \geq 0 \mid U^k \cdot \HFred(Y) = 0 \}.
\]
It is precisely the failure of \eqref{eq:NiWu} to hold in general that makes it possible to obstruct
the existence of PL disks and spheres.
\end{remark}

\begin{remark} \label{rmk:0-framing}
There is also an obstruction to the existence of a PL sphere in the case where $n=0$, although we do
not know of any actual example where it is effective. If $Y$ is any $3$-manifold with vanishing
triple cup product on $H^1(Y)$, and $\spincs$ is any torsion spin$^c$ structure on $Y$, then there
are two relevant invariants to consider: the untwisted ``bottom'' $d$ invariant $d_b(Y,\spincs)$
defined by Ozsv\'ath and Szab\'o \cite{OSabsgr} (see also \cite{LRS}), and the totally twisted $d$
invariant $\ul d(Y,\spincs)$ defined by Behrens and Golla \cite{BehrensGollaCorrection}. These
invariants are both preserved under spin$^c$ homology cobordism, and they satisfy $\ul d(Y,\spincs)
\le d_b(Y,\spincs)$ \cite[Proposition 3.8]{BehrensGollaCorrection}. We do not know of any
$3$-manifold for which this inequality is strict.

For any knot $K \subset S^3$, Behrens and Golla showed that $\ul d(S^3_0(K), \spincs_0) = d_b(S^3_0(K), \spincs_0)$, where $\spincs_0$ denotes the unique torsion spin$^c$ structure \cite[Example 3.9]{BehrensGollaCorrection}. Just as in the proof of Theorem \ref{thm:obstruction}, it follows that if $X$ is a smooth $4$-manifold with the homology of $S^2$ and vanishing intersection form, and the generator of $H_2(X)$ can be represented by a PL sphere, then $\ul d(\partial X, \spincs_0) = d_b(\partial X, \spincs_0)$.
\end{remark}

\section{Construction}

We now describe a family of $4$-manifolds homotopy equivalent to $S^2$ which fail to satisfy the conclusion of Theorem \ref{thm:obstruction}.

\begin{figure}
\subfigure[]{
\labellist
 \pinlabel $0$ [t] at 17 8
 \pinlabel $m+2$ [t] at 60 8
\endlabellist
\includegraphics{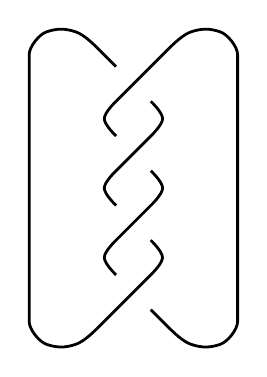}
\label{subfig:Qm-T24}}
\subfigure[]{
\labellist
 \pinlabel $0$ [t] at 17 8
 \pinlabel $m-2$ [t] at 60 8
\endlabellist
\includegraphics{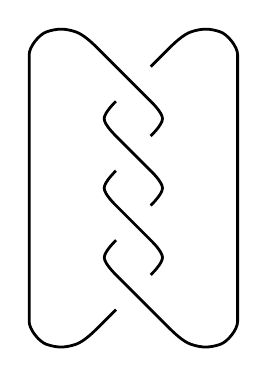}
}
\subfigure[]{
\labellist
 \pinlabel $-2$ [b] at 32 95
 \pinlabel $-1$ [b] at 72 95
 \pinlabel $m$ [b] at 112 95
 \pinlabel $-2$ [t] at 72 8
\endlabellist
\includegraphics{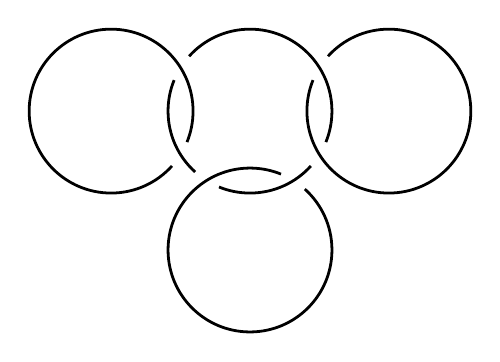}
\label{subfig:Qm-seifert}}
\caption{Three surgery descriptions of $Q_m$.}
\label{fig:Qm}
\end{figure}
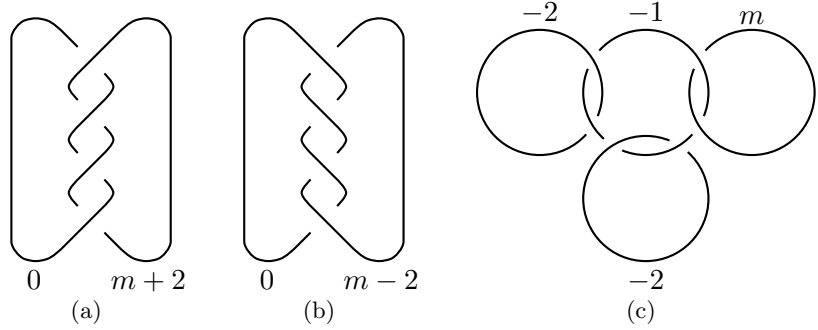

For any integer $m$, let $Q_m$ denote the total space of a circle bundle over $\RP^2$ with Euler number $m$. This is a rational homology sphere with
\[
H_1(Q_m) \cong
\begin{cases}
  \Z/2 \oplus \Z/2 & m \text{ even} \\
  \Z/4 & m \text{ odd}.
\end{cases}
\]
The manifold $Q_m$ can be described by any of the surgery diagrams in Figure \ref{fig:Qm}.

For any $m$, Doig \cite[Section 3]{Doig} proved that the $d$ invariants of $Q_m$ in the four
spin$^c$ structures are
\begin{equation}\label{eq:d(Qm)}
\left\{\frac{m+2}{4}, \frac{m-2}{4}, 0, 0\right\}.
\end{equation}
(See also the work of Ruberman, Strle, and the first author \cite[Theorem 5.1]{LRS}.)


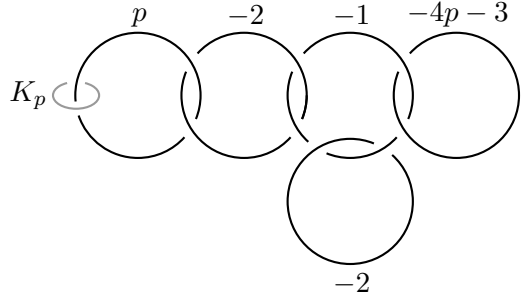
\begin{figure}
\labellist
 \pinlabel $p$ [b] at 41 95
 \pinlabel $-2$ [b] at 81 95
 \pinlabel $-1$ [b] at 121 95
 \pinlabel $-4p-3$ [b] at 161 95
 \pinlabel $-2$ [t] at 121 8
 \pinlabel $K_p$ [r] at 9 72
\endlabellist
\includegraphics{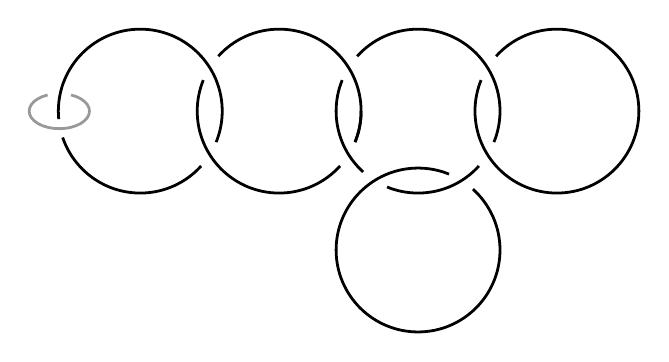}
\caption{Surgery description of the Brieskorn sphere $Y_p$.  The knot $K_p$ represents a singular fiber in a Seifert fibration on $Y_p$.} \label{fig:plumbing}.
\end{figure}

For each integer $p$, let $Y_p$ be the $3$-manifold given by the surgery diagram in Figure \ref{fig:plumbing}, which naturally bounds a plumbed $4$-manifold.
It is easy to check that $Y_p$ is the Seifert fibered homology sphere
\[
Y_p \cong
\begin{cases}
\Sigma(2,-(2p+1),-(4p+3)) & p < -1 \\
S^3 & p=-1,0 \\
-\Sigma(2,2p+1,4p+3) & p > 0.
\end{cases}
\]
(Our convention is that for pairwise relatively prime integers $a,b,c>0$, the Brieskorn sphere $\Sigma(a,b,c)$ is oriented as the boundary of a positive-definite plumbing. Note, however, that the plumbing shown in Figure \ref{fig:plumbing} is indefinite.)

Let $K_p \subset Y_p$ be the knot obtained as a meridian of the $p$-framed surgery curve, shown in Figure~\ref{fig:plumbing}. In the cases $p=-1$ or $p=0$, where $Y_p \cong S^3$, $K_p$ is the unknot or the right-handed trefoil, respectively; otherwise, $K_p$ is the singular fiber of order $2p+1$. The $0$-framing on this curve (viewed as a knot in $S^3$) corresponds to the $+4$ framing on $K_p$ (as a knot in $Y_p$). Performing surgery using this framing produces $Q_{-4p-3}$, since we can cancel the $p$-framed component with its $0$-framed meridian to produce Figure \ref{subfig:Qm-seifert} with $m=-4p-3$.



We are now able to construct the spineless four-manifolds claimed in Theorem~\ref{thm:spineless}.
Define the four-manifold $W_p$ obtained by taking $(Y_p - B^3) \times [0,1]$, which has boundary $Y_p
\conn {- Y_p}$, and attaching a $+4$-framed 2-handle along the knot $K_p \times \{1\}$.  The boundary of $W_p$
is $Q_{-4p-3} \conn {- Y_p}$; denote this three-manifold by $M_p$.

\begin{proposition}
For each $p$, the manifold $W_p$ is homotopy equivalent to $S^2$.
\end{proposition}
\begin{proof}
First, notice that $(Y_p - B^3) \times [0,1]$ is an integer homology ball, so after attaching the
2-handle, $W_p$ has the same homology as $S^2$.  To show that $W_p$ is simply-connected (and hence homotopy equivalent to $S^2$), it is sufficient to show that the homotopy class of $K_p$ normally generates $\pi_1(Y_p)$.  This is obvious in the case that $p = -1, 0$ as $Y_p = S^3$.  The following lemma proves this claim in the remaining cases.
\end{proof}

\begin{figure}
\labellist
 \pinlabel $\dfrac{p}{p'}$ [r] at 10 79
 \pinlabel $\dfrac{q}{q'}$ [l] at 135 79
 \pinlabel $\dfrac{r}{r'}$ [r] at 50 34
 \pinlabel $e$ [bl] at 80 102
 \small
 \pinlabel $x$ [b] at 32 112
 \pinlabel $h$ [b] at 72 112
 \pinlabel $y$ [b] at 112 112
 \pinlabel $z$ [t] at 72 8
\endlabellist
\includegraphics{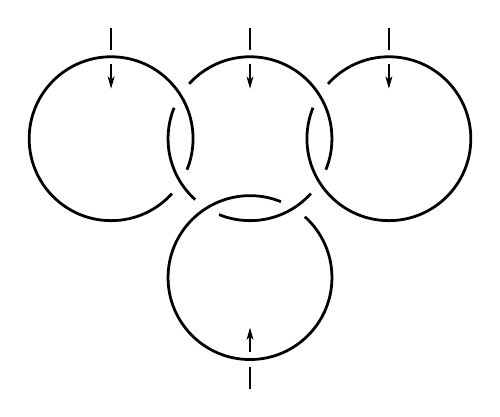}
\caption{Surgery description of $\Sigma(p,q,r)$, along with generators for $\pi_1$.} \label{fig:seifert-pi1}
\end{figure}
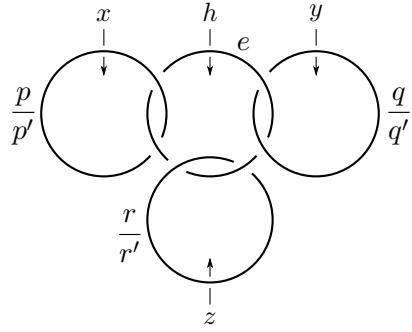

\begin{lemma}\label{lemma:brieskornweightone}
For any pairwise relatively prime integers $p,q,r$, the fundamental group of the Brieskorn sphere $\Sigma(p,q,r)$ is normally generated by any of the singular fibers.
\end{lemma}
\begin{proof}
Recall that if $\Sigma(p,q,r) = S^2(e;(p,p'), (q,q'), (r,r'))$, then
\begin{equation} \label{eq:brieskorn-pi1}
\pi_1(\Sigma(p,q,r)) = \langle x,y,z,h \mid h \text{ central}, \, x^{p} h^{p'} = y^{q} h^{q'} = z^{r} h^{r'} = xyzh^e = 1\rangle.
\end{equation}
To see this presentation, we consider the standard surgery description for $\Sigma(p,q,r)$ as in Figure \ref{fig:seifert-pi1}.  The complement of the surgery link $L$ has
\[
\pi_1(S^3 - L) = \langle x, y, z, h \mid h \text{ central} \rangle.
\]
Here, $x,y,z$ represent meridians of the three parallel curves while $h$ represents the fiber direction.  The four additional relators in \eqref{eq:brieskorn-pi1} represent the longitudes filled by the Dehn surgeries.


Without loss of generality, we consider the singular fiber of order $p$, which is the core of the Dehn surgery on the leftmost component in Figure \ref{fig:seifert-pi1}. This curve is represented in $\pi_1(\Sigma(p,q,r))$ by $x^a h^b$, where $a,b$ are any integers such that $\abs{bp - ap'} = 1$. Thus, we must show that the quotient $G = \pi_1(\Sigma(p,q,r)) / \langle\! \langle x^a h^b \rangle \!\rangle$ is trivial. Because $x$ and $h$ commute and $\abs{bp-ap'} = 1$, the subgroup of $G$ generated by $x$ and $h$ is the same as the subgroup generated by $x^a h^b$ and $x^p h^{p'}$.  Therefore, $x = h = 1$ in $G$, so
\[
G \cong \langle y, z \mid y^q = z^r = yz = 1\rangle.
\]
Since $q$ and $r$ are relatively prime, this implies that $G$ is the trivial group.  Consequently, the singular fibers normally generate the fundamental group of $\Sigma(p,q,r)$.
\end{proof}

The following proposition now establishes Theorem~\ref{thm:spineless}; specifically, it shows that
the manifolds $W_p$ are spineless for $p \not\in \{-2,-1,0\}$. (Both $W_{-1}$ and $W_0$ contain
spines since they are obtained by attaching a $2$-handle to the $4$-ball; we do not know whether
$W_{-2}$ has a spine.)
\begin{proposition}
If $M_p$ bounds a compact, smooth, oriented $4$-manifold $X$ with $H_*(X) \cong H_*(S^2)$ in which
a generator of $H_2(X)$ can be represented by a PL $2$-sphere, then $p \in \{-2,-1,0\}$.
\end{proposition}
\begin{proof}
Suppose $M_p$ bounds a compact, smooth, oriented $4$-manifold $X$ with $H_*(X) \cong H_*(S^2)$.
Observe that the four $d$ invariants of $M_p$ are equal to those of $Q_{-4p-3}$ minus the even
integer $d(Y_p)$. To be precise, label the four spin$^c$ structures on $M_p$ by $\spincs_0, \dots,
\spincs_3$ according to Convention \ref{conv:spinc}. By \eqref{eq:d-mod2}, we deduce that the
intersection form of $X$ must be positive-definite, and
\[
d(M_p, \spincs_0) \equiv \frac34, \quad d(M_p, \spincs_1) = d(M_p, \spincs_3) \equiv 0, \quad
d(M_p, \spincs_2) \equiv \frac74 \pmod{2\Z}.
\]
(If the intersection form were negative-definite, the $d$ invariants of $\spincs_0$ and $\spincs_2$
would be congruent to $\frac54$ and $\frac14$ respectively, which would violate \eqref{eq:d(Qm)}.)
These congruences enable us to identify which of the two self-conjugate spin$^c$
structures is $\spincs_0$ and which is $\spincs_2$. Specifically, when $p$ is odd, we have
\begin{align*}
d(M_p, \spincs_0) &= -d(Y_p) - \frac{4p+1}{4}  \\
d(M_p, \spincs_1) = d(M_p, \spincs_3) &= -d(Y_p) \\
d(M_p, \spincs_2) &= -d(Y_p) - \frac{4p+5}{4}.
\end{align*}
By Theorem \ref{thm:obstruction}, if there is a PL sphere representing a generator of
$H_2(X)$, then:
\begin{align*}
-\frac{4p+1}{4} = d(M_p, \spincs_0) - d(M_p, \spincs_1) &= \frac34 \text{ or } {-\frac54} \\
\frac{4p+5}{4} = d(M_p, \spincs_1) - d(M_p, \spincs_2) &= \frac14 \text{ or } {-\frac74}
\end{align*}
These two equations imply that $p=-1$.

Similarly, when $p$ is even, the roles of $\spincs_0$ and $\spincs_2$ are exchanged, and we deduce
that $p$ equals either $-2$ or $0$.
\end{proof}

\begin{remark}
In \cite{Doig}, Doig computed the $d$ invariants of $Q_m$ and used these to show that many of the $Q_m$ cannot be obtained by surgery on a knot in $S^3$.  Our arguments further show that $Q_m$ cannot be integrally homology cobordant to surgery on a knot. While Doig's arguments use $d$ invariants, which are homology cobordism invariants, they also rely on the fact that the $Q_m$ are L-spaces, which is not a property that is preserved under homology cobordism.
\end{remark}

\begin{remark}
For any $k>1$, one can modify the construction above to obtain spineless $4$-manifolds $X$ with
$H_1(\partial X) \cong \Z/k^2$. Let $Q_{k,m}$ be the manifold obtained by $(0,m+k)$ surgery on the
$(2,2k)$ torus link. (Using our previous notation, $Q_m = Q_{2,m}$, as seen in Figure
\ref{subfig:Qm-T24}.) Then $\abs{H^2(Q_{k,m})} = k^2$, and $H^2(Q_{k,m})$ is cyclic iff
$\gcd(k,m)=1$. Since $Q_{k,m}$ bounds a rational homology ball, the $d$ invariants of $k$ of the
$k^2$ spin$^c$ structures on $Q_{k,m}$ vanish. On the other hand, the exact triangle relating the
Heegaard Floer homologies of $S^1 \times S^2$, $Q_{k,m}$, and $Q_{k,m+1}$ shows that the $d$
invariants of the remaining spin$^c$ structures vary roughly linearly in $m$. In particular, the
differences between $d$ invariants of adjacent spin$^c$ structures can be arbitrarily large.
Moreover, one can realize $Q_{k,m}$ (for appropriate $m$) as surgery on a fiber in a Brieskorn
sphere; the result then follows as above.

We do not know of any instances where Theorem \ref{thm:obstruction} obstructs the existence of a PL sphere when $n$ is not a perfect square.
\end{remark}

\bibliographystyle{amsalpha}
\bibliography{bib}

\end{document}